\documentclass[11pt]{amsart}
\usepackage{lmodern}
\usepackage[T1]{fontenc}
\usepackage{amssymb, amsmath, amsfonts,amsthm,enumerate}
\usepackage[mathscr]{eucal}
\usepackage[raiselinks=false,colorlinks=true,citecolor=blue,urlcolor=blue,linkcolor=blue,bookmarksopen=true,pdftex]{hyperref}
\usepackage{graphicx}
\usepackage{setspace}
\usepackage[T1]{fontenc}
\usepackage[margin=1.0in]{geometry}
\usepackage{color}

\def\scfig #1 #2 {\resizebox{#2}{!}{\includegraphics{#1}}}

\numberwithin{equation}{section}

\def\degree{\operatorname{degree}}

\def\wt{\operatorname{wt}}
\def\Q{\mathbb Q}
\def\N{\mathbb N}

\newcommand{\lm}{{\lambda}}
\theoremstyle{plain}

\newtheorem{theorem}[equation]{Theorem}
\newtheorem{corollary}[equation]{Corollary}
\newtheorem{proposition}[equation]{Proposition}
\newtheorem{lemma}[equation]{Lemma}

\theoremstyle{definition}

\newtheorem{example}[equation]{Example}

\newtheorem{remark}[equation]{Remark}
\newtheorem{discussion}[equation]{Discussion}
\newenvironment{discussionbox}[1][]{%
    \begin{discussion}[#1]\pushQED{\qed}}{\popQED \end{discussion}}

\newcommand{\bc}{\mathbb{C}}

\begin{document}
 \title[Rational homotopy of maps between complex Grassmann manifolds]{Rational homotopy of maps between certain complex Grassmann manifolds}

\author[Prateep Chakraborty \and  Shreedevi K. Masuti]{Prateep Chakraborty* \and  Shreedevi K. Masuti**}
\newcommand{\acr}{\newline\indent}
\address {\llap{*\,} Stat-Math Unit \acr Indian Statistical Institute \acr 8th Mile Mysore Road \acr Bangalore 560 059 \acr India.}
\email{chakraborty.prateep@gmail.com}
\address{\llap{**\,}Institute of Mathematical Sciences \acr 
IV Cross Road, CIT Campus \acr 
Taramani \acr  
Chennai 600113 \acr 
India.}
\email{shreedevikm@imsc.res.in, masuti.shree@gmail.com}

\thanks{The second author thanks Indian Statistical Institute Bangalore for providing local hospitality during the course of writing this paper. She also thanks Department of Atomic Energy, Government of India, for providing funding for her post doctoral studies during which this work is done.}
\subjclass[2010]{Primary 55S37, 13A02; Secondary 57T15}
\keywords{Grassmann manifold, homotopy class of maps, graded algebra homomorphism, cohomology algebra}

\begin{abstract}
 Let $G_{n,k}$ denote the complex Grassmann manifold of $k$-dimensional vector subspaces of $\mathbb{C}^n$. Assume $l,k\le \lfloor n/2\rfloor$. We show that, for sufficiently large $n$, any continuous map $h:G_{n,l}\to G_{n,k}$ is rationally null homotopic if $(i)~ 1\le k< l,$ $(ii)~2<l<k< 2(l-1)$, $(iii)~1<l<k$, $l$ divides $n$ but $l$ does not divide $k$. 
 \end{abstract}
\maketitle

\section{Introduction}
An important problem in homotopy theory is to classify, up to homotopy, maps between two given topological spaces.
This is often a difficult problem even when the spaces involved are  well-behaved 
spaces such as simply connected compact smooth manifolds.   The determination of the set of {\it rational} homotopy classes 
of maps between such spaces is more tractable, thanks to the work of Sullivan \cite{sullivan}.   In the special 
case when the spaces involved are {\it rationally formal} in the sense of rational homotopy theory 
this problem is reduced to studying the graded algebra homomorphisms between their rational (singular) cohomology 
algebras. See \cite[p.156]{FHT}.  

Denote by $G_{n,k}$ the complex Grassmann manifold of $k$-dimensional vector subspaces of $\bc^n$.  
It can be identified with the homogeneous space $U(n)/(U(k)\times U(n-k))$ where $U(n)$ denotes the group of $n\times n$ unitary matrices.   In particular it is a simply connected Hermitian symmetric space of complex dimension $k(n-k)$.  Thus $G_{n,k}$ 
is a rationally formal space \cite[p.162]{FHT}.    Our aim in this paper is to establish the following result.  Since $G_{n,k}\cong G_{n,n-k}$, we assume that $1\le k\le \lfloor n/2\rfloor$.

\begin{theorem}\label{main theorem}  Let $l,k\le \lfloor n/2\rfloor$. 
(i) Suppose that $n\ge 2l^2+l-2,~ 1\le k< l$.  Then any continuous map $h:G_{n,l}\to G_{n,k}$ is rationally null 
homotopic.  \\
(ii) Suppose that $n\ge 3k^2-2$.  Then any continuous map $h:G_{n,l}\to G_{n,k}$ is rationally null homotopic 
in any of the following cases: (a)  $2<l<k< 2(l-1)$, (b) $1<l<k$, $l$ divides $n$ but $l$ does not divide $k$. 
\end{theorem}
 As a consequence we see that the set $[G_{n,l},G_{n,k}]$ of homotopy classes of maps between the 
indicated Grassmann manifolds is finite when $n,k,l$ are as in the above theorem. The first part of the theorem implies that for a fixed $l$ and $k<l$, with at most {\it finitely} many exceptions the set $[G_{n,l},G_{n,k}]$ is finite. We shall in fact establish a stronger statement than Case $(ii)(b)$ of Theorem \ref{main theorem} covering many more cases. See Remark \ref{other result} for the detailed discussion of this paragraph.

We now recall briefly the history of the problem of classifying maps between complex Grassmann manifolds. 
It was shown by Friedlander \cite{friedlander} that the complex Grassmann manifolds admit self-maps 
of arbitrarily high degree. 
The classification of endomorphisms of the cohomology algebras of the complex Grassmann manifolds has been studied in \cite{oniel}, 
\cite{brewster-homer}, \cite{gh}, \cite{hoffman}, with applications to fixed point property.  See also \cite{hoffman-homer} in which self-maps of certain complex flag manifolds are considered.  
Study  
of continuous maps  via their induced homomorphisms in mod $2$ cohomology algebras between {distinct} real Grassmann manifolds was initiated in \cite{ks}.   In \cite{rs}, degrees of maps between distinct oriented real Grassmann manifolds of the {\it same} dimension was determined.  The same problem for distinct complex Grassmann manifolds was considered in \cite{ss}.  
We point out that,  Paranjape and Srinivas \cite{ps} showed that any non-constant {\it algebraic morphism} between two complex Grassmann manifolds of 
the same dimension 
is an isomorphism provided that the target is {\it not} the complex projective space.  

One could consider the more general problem of classifying continuous maps $G_{m,l}\to G_{n,k}$. Since $G_{n,k}$ is a complex flag manifold; this problem of classification, up to rational homotopy, boils down to classifying graded algebra homomorphisms between the cohomology algebras with rational coefficients $H^*(G_{n,k};\Q)$ and $H^*(G_{m,l};\Q)$ (see Theorem \ref{rigid}).
Since $H^*(G_{m-1,l};\Q)$ is a quotient of $H^*(G_{m,l};\Q)$, the vanishing of any graded algebra homomorphism from $H^*(G_{n,k};\mathbb Q)$ to $H^*(G_{n,l};\mathbb Q)$ in positive degrees when $m>n$ does imply the same conclusion for graded algebra homomorphisms from $ H^*(G_{n,k};\mathbb Q)$ to $H^*(G_{m,l};\mathbb Q)$. 
Partial results when $k<l$ and $m-l>n-k$ were obtained in \cite{cs}.      
The nature of the structure of the 
rational cohomology algebra of the complex Grassmann manifolds still makes the problem quite challenging, 
and perhaps explains why the problem still remains unsolved for general values of $m,n,k,l$.

\section{Preliminaries}\label{Preli}
Let $\gamma_{n,k}$ be the ``tautological bundle'' over $G_{n,k}$, whose fiber over a point $V\in G_{n,k}$ is the $k$- dimensional complex vector space $V$.  This vector bundle $\gamma_{n,k}$ is a rank $k$-subbundle of the rank $n$ trivial bundle $\varepsilon^n$.  Let $\gamma_{n,k}^\bot$ denote the orthogonal complement of $\gamma_{n,k}$ in $\varepsilon^n$ (with respect to a hermitian metric on $\mathbb{C}^n$).  Let $x_i:=c_i(\gamma_{n,k})\in H^{2i}(G_{n,k};\mathbb{Z})$ be the i-th Chern class of $\gamma_{n,k}$, $1\leq i\leq k$.  Denoting the total Chern class of a vector bundle $\eta$ by $c(\eta)$ we see that $c(\gamma_{n,k})c(\gamma_{n,k}^\bot)=1$.  We write $c=1+x_1+\dots+x_k$.\\
From \cite{borel} and \cite[Theorem 2.1]{hoffman}, we know that the cohomology algebra of $G_{n,k}$ is 
$$H^*(G_{n,k};\Q)=\frac{\Q[x_1,\dots,x_k]}{\mathcal{I}_{n,k}},$$
where $\degree x_i=2i$ and $\mathcal{I}_{n,k}$ is the ideal generated by $(c^{-1})_{n-k+j}$, $1\leq j\leq k$. Here $(c^{-1})_j$ denotes the homogeneous part of the formal inverse of $c$ having (total) degree $2j$.
The ideal $\mathcal{I}_{n,k}$ contains all the elements $(c^{-1})_j$, $j\geq n-k+1$ and 
$$c^{-1}=1+(c^{-1})_1+(c^{-1})_2+\dots+(c^{-1})_{n-k}\in H^*(G_{n,k};\Q)$$
is in fact the total Chern class of $\gamma_{n,k}^\bot$.\\

We denote $(c^{-1})_{n-k+j},$ $1\leq j\leq k,$ by $R_j$.
For ${\bf n}=(n_1,\ldots,n_k) \in \mathbb N^k$, we set 
$|{\bf n}|=n_1+\cdots+n_k, \wt {\bf n}=\sum_{i=1}^kin_i$, ${\bf x}^{\bf n}=x_1^{n_1}\cdots x_k^{n_k}$ 
and $c_{\bf n}=\frac{|{\bf n}|!}{n_1!n_2!\ldots n_k!}$. By \cite[p.174]{gh}, we can write $R_j$ in the following form
\begin{eqnarray}\label{expression of R_i}
 R_j=\sum_{{\bf n}\in\mathbb{N}^k,{\wt {\bf n}}=n-k+j}(-1)^{|{\bf n}|}c_{\bf n}{\bf x}^{\bf n}.
\end{eqnarray}
We shall use Expression (\ref{expression of R_i}) of $R_j$ in our proofs. We recall the following well-known facts about the cohomology algebra of $G_{n,k}$, which will be used in the proofs of our results: \\

\noindent (1) The homogeneous polynomials $R_1,\ldots,R_k \in \mathcal I_{n,k}$ form a {\it regular sequence} in the polynomial algebra $\Q[x_1,x_2,\dots,x_k]$, which means that $(R_1,\ldots,R_k) \neq \Q[x_1,x_2,\dots,x_k]$, $R_1\neq0$ and $R_{j+1}$ is not a zero-divisor in $\frac{\Q[x_1,x_2,\dots,x_k]}{(R_1,R_2,\dots,R_j)}$, for $1\leq j\leq k-1$.  See \cite{bott-tu}.\\

\noindent (2) The natural imbeddings $\hat{i}:G_{n,k}\hookrightarrow G_{n+1,k}$ and $\hat{j}:G_{n,k}\hookrightarrow G_{n+1,k+1}$ induce surjections $\hat{i}^*:H^*(G_{n+1,k};\Q)\to H^*(G_{n,k};\Q)$ and $\hat{j}^*:H^*(G_{n+1,k+1};\Q) \to H^*(G_{n,k};\Q)$, where $\hat{i}^*(x_r+\mathcal I_{n+1,k})=x_r+\mathcal I_{n,k}=\hat{j}^*(x_r+\mathcal I_{n+1,k+1})$ for $1\leq r\leq k$ and $\hat{j}^*(x_{k+1}+\mathcal I_{n+1,k+1})=0$.  The homomorphism $\hat{i}^*$ induces isomorphisms in cohomology in dimensions up to $2(n-k)$ and $\hat{j}^*$ induces isomorphisms in cohomology in dimensions up to $2k$.

Next we shall prove a lemma, which we will need to prove the main results of this paper. Throughout this paper, by saying ``a graded algebra homomorphism is trivial'', we shall mean that it vanishes in positive degrees.
\begin{lemma}\label{vanishing of Ri}
Let $n,l,k\geq1$ and $\mathcal{I}_{n,k}=(R_1,\ldots,R_k)$. Let $\phi: \Q[x_1,\ldots,x_k] \rightarrow \Q[y_1,\ldots,y_l] $ be a graded algebra homomorphism 
such that $\phi(R_i)=0$ for all $i=1,\ldots,k.$ Then $\phi$ is trivial.
\end{lemma}
\begin{proof}
Since $H^*(G_{n,k};\Q)$ is finite dimensional as $\Q$-vector space, there exists an integer $N_i$ such that $x_i^{N_i} \in \mathcal I_{n,k} $ for all $i$, $1\leq i\leq k.$ Hence $({\phi}(x_i))^{N_i}=\phi(x_i^{N_i})=0$ which implies that ${\phi}(x_i)=0$ 
for all $i=1,\ldots, k.$ Hence $\phi$ is trivial.
\end{proof}

We recall the following proposition from \cite{gh} which will be used frequently in this paper.

\begin{proposition} \cite[Proposition 1]{gh} \label{proposition of gh}
 Let $N,m_0,n_0$ be integers such that 
 $$N \geq p(p-1)+n_0p+m_0(p-1).$$
 Then there exist integers $m\geq m_0 $ and $n\geq n_0$ such that 
 $$m(p-1)+np=N.$$
\end{proposition}

\section{Graded algebra homomorphisms from $H^*(G_{n,k};\Q)$ to $H^*(G_{n,l};\Q)$ for $k <l$}
In this section we consider the possible graded algebra homomorphisms from $H^*(G_{n,k};\Q)$ to $H^*(G_{n,l};\Q)$ where $k <l$.  
We prove that any such graded algebra homomorphism is trivial for sufficiently large $n$ (Theorem \ref{1st theorem}). 
In order to prove this we show that $x_{l-1}x_l$ occurs in
every nonzero homogeneous polynomial of degree at most $2n$ in $\mathcal I_{n,l}$ for sufficiently large $n$ (Corollary \ref{existence of xl}). The restriction on $n$ is needed in the course of proving this result.

We also prove that if $P_i,1\leq i\leq l,$ is a nonzero homogeneous polynomial in 
 $\mathcal I_{n,l}$ of degree $2(n-l+i)$ then $P_i$ contains $x_{i-1},\ldots,x_l$ for 
 $n \geq 2l^2+l-2$. To obtain these results, we first prove the following proposition. We denote the vector 
 $(0,\ldots,1,\ldots,0)\in \mathbb N^l,$ where $1$ occurs at the $i$th position and $0$ elsewhere, by ${\bf e_i}.$ 
 Let $1\leq c_1,c_2,l$ be integers. For $i=1,\ldots,l$ and ${\bf s}\in Q(i),$ we set 
 \begin{eqnarray*}
 Q(i)&:=&\{{\bf n}\in \mathbb N^l |1\leq \wt {\bf n}\leq i-1\},\\
 D_i(c_1,c_2,n,l)&:=&\{x_{l-1}^{m_1}x_l^{m_2}: m_1 \geq c_1,   m_2 \geq c_2 
 \mbox{ satisfying }(l-1)m_1+lm_2=(n-l+i)\} \mbox{ and }\\
D_{i,{\bf s}}(c_1,c_2,n,l)&:=&\{{\bf x}^{\bf s}x_{l-1}^{m_1}x_l^{m_2}:
 m_1 \geq c_1, m_2 \geq c_2   
 \mbox{ satisfying }(l-1)m_1+lm_2=(n-l+i)-\wt {\bf s}\}.
 \end{eqnarray*}
We shall use $D_i$ instead of $D_i(c_1,c_2,n,l)$ and $D_{i,{\bf s}}$ instead of $D_{i,{\bf s}}(c_1,c_2,n,l)$ to simplify notations.
 \begin{proposition}\label{existence of monomials}
 Let $1\leq l,c_1,c_2 $ be integers such that $n-l+1\geq l(l-1)+(c_1+l)(l-1)+c_2l.$ Let $1 \leq i \leq l$ be fixed. Then 
 \begin{enumerate}
 \item \label{existence of monomials part1} the sets $D_i$ and $D_{i,{\bf s}}$ are non-empty;
 \item \label{existence of monomials part2} every nonzero homogeneous polynomial 
of degree $2(n-l+i)$ in $\mathcal I_{n,l}$ contains at least one monomial from the set 
$$\varLambda_i:=D_i \bigcup\Big(\bigcup_{{\bf s} \in Q(i)} D_{i, {\bf s}}\Big)$$
 with nonzero coefficient.
 \end{enumerate}
\end{proposition}
\begin{proof}
 \eqref{existence of monomials part1}: Follows from Proposition \ref{proposition of gh}.  

 \eqref{existence of monomials part2}:
Let $I=\mathcal I_{n,l}.$ For fixed $1 \leq i \leq l,$ let $0\neq P_i \in I$ be a homogeneous polynomial of degree $2(n-l+i).$ Then 
\begin{eqnarray} \label{expression for Pi}
 P_i=N_iR_i+\sum_{{\bf s} \in Q(i)} N_{i,{\bf s}} {\bf x}^{\bf s}R_{i-\wt {\bf s}}
\end{eqnarray}
for some $N_i,N_{i,{\bf s}} \in \Q$.  
Suppose $P_i$ does not contain any monomial from the set $\varLambda_i.$ \\
\noindent \underline{Case 1:} $i<l$\\
Since $n-l+i\geq n-l+1\geq l(l-1)+(c_1+l)(l-1)+c_2l,$ 
by Proposition \ref{proposition of gh}, there exist integers $m_1\geq c_1+l,m_2 \geq c_2$ such that 
$$n-l+i=m_1(l-1)+m_2l.$$ 
Observe that $\max\{j:x_j \mbox{ divides } {\bf x}^{\bf s}\mbox{ for some }{\bf s}\in Q(i)\}=i-1$. 
Therefore comparing the coefficients of $x_{l-1}^{m_1}x_l^{m_2}$ in 
(\ref{expression for Pi}), we get that $N_i=0$. \\
In order to show that $N_{i,{\bf s}}=0$ for all ${\bf s} \in Q(i)$ we order the set $Q(i)$ 
lexicographically and induct on ${\bf s}$. This means that ${\bf n}< {\bf m}$ in $Q(i)$ if the 
first nonzero entry of ${\bf m}-{\bf n}$ is positive. This defines a total order on $Q(i)$. 
Note that the vector ${\bf e_{i-1}} \in Q(i)$ is the smallest element in $Q(i)$. 
First we show that $N_{i,{\bf e_{i-1}}}=0$. Since $n-l+1 \geq  l(l-1)+(c_1+l)(l-1)+c_2l$, again 
by Proposition \ref{proposition of gh}, there exist integers $m_3\geq c_1+l,m_4 \geq c_2$ such that 
$$n-l+1=m_3(l-1)+m_4l.$$
Comparing the coefficients of $x_{i-1}x_{l-1}^{m_3}x_l^{m_4}$ in (\ref{expression for Pi}), 
we get that $N_{i,{\bf e_{i-1}}}=0$. \\
Let ${\bf s} \in Q(i)$ be fixed and assume that $N_{i,{\bf t}}=0$ for ${\bf t}<{\bf s}$. 
Since $n-l+i-\wt {\bf s} \geq n-l+1 \geq  l(l-1)+(c_1+l)(l-1)+c_2l$, by Proposition \ref{proposition of gh}, 
there exist integers $m_5\geq c_1+l,m_6 \geq c_2$ such that
$$n-l+i-\wt {\bf s}=m_5(l-1)+m_6l.$$
Then the monomial ${\bf x}^{\bf s}x_{l-1}^{m_5}x_l^{m_6}$ occurs in ${\bf x}^{\bf t}R_{i-\wt {\bf t}}$ implies that ${\bf t} \leq {\bf s}$. Hence comparing the coefficients of ${\bf x}^{\bf s}x_{l-1}^{m_5}x_l^{m_6}$ in (\ref{expression for Pi}) and using the induction 
hypothesis we get that $N_{i,{\bf s}}=0$. Therefore $N_{i,{\bf s}}=0$ for all ${\bf s} \in Q(i)$ 
which implies that $P_i=0$, a contradiction. 
Hence $P_i $ must contain at least one monomial from the set $\varLambda_i$.

\noindent \underline{Case 2 :} $i=l$ \\
We have
\begin{eqnarray} \label{expression for Pl}
 P_l=N_lR_l+\sum_{{\bf s} \in Q(l)} N_{l,{\bf s}} {\bf x}^{\bf s}R_{l-\wt {\bf s}}.
\end{eqnarray} 
Since $n \geq n-l+1\geq l(l-1)+(c_1+l)(l-1)+c_2l,$ 
by Proposition \ref{proposition of gh}, there exist integers $m_7\geq c_1+l,m_8 \geq c_2$ such that
$$n=m_7(l-1)+m_8l.$$
Comparing the coefficients of $x_{l-1}^{m_7}x_l^{m_8}$ in Equation (\ref{expression for Pl}), we get that
\begin{eqnarray*}
 (-1)^{m_7+m_8} \frac{(m_7+m_8)!}{m_7!~m_8!}N_{l}+(-1)^{m_7+m_8-1}\frac{(m_7+m_8-1)!}{(m_7-1)!m_8!}
 N_{l,{\bf e_{l-1}}}=0.
 \end{eqnarray*}
Hence 
\begin{eqnarray} \label{equation1}
 \frac{m_7+m_8}{m_7}N_l-N_{l,{\bf e_{l-1}}}=0.
\end{eqnarray}
Since $n=(m_7-l)(l-1)+(m_8+l-1)l$, comparing the coefficients of $x_{l-1}^{m_7-l}x_l^{m_8+l-1}$ in Equation (\ref{expression for Pl}), we get that 
\begin{eqnarray} \label{equation2}
 \frac{m_7+m_8-1}{m_7-l} N_l-N_{l,{\bf e_{l-1}}}=0.
\end{eqnarray}
Comparing Equations (\ref{equation1}) and (\ref{equation2}), we get that $N_l=0=N_{l,{\bf e_{l-1}}}$. \\
To show that $N_{l,{\bf s}}=0$ for all ${\bf s} \in Q(l)$, order the set $Q(l)$ 
lexicographically and induct on ${\bf s}$. The same argument, as in Case 1, shows that $N_{l,{\bf s}}=0$ for all ${\bf s}\in Q(l)$.
Therefore $P_l$ must contain a monomial 
from the set $\varLambda_l.$
\end{proof}
Taking $c_1=c_2=1$ in Proposition \ref{existence of monomials}, we get
\begin{corollary} \label{existence of xl}
 Every nonzero homogeneous polynomial of degree at most $2n$ in $\mathcal I_{n,l}$ 
contains a monomial, with nonzero coefficient, that is divisible by $x_{l-1}x_l$ if $n \geq 2l^2+l-2.$
 \end{corollary}
 
As a consequence we prove the following theorem. This theorem can be viewed as a general property of the cohomology algebra of a complex Grassmann manifold. 
\begin{theorem}\label{interesting}
For fixed $1 \leq i \leq l$, let $P_i \in \mathcal I_{n,l}$ be a nonzero homogeneous polynomial of degree $2(n-l+i)$. Assume that $n \geq 2l^2+l-2$. 
Then, for $k=i-1,\ldots,l$, $P_i$ contains a monomial ${\bf x}^{\bf n}$ (${\bf n}$ depending on $k$), 
with nonzero coefficient, that is divisible by $x_k.$ 
\end{theorem}
\begin{proof}
 We induct on $l$. For $l=1,2$ the result follows from Corollary \ref{existence of xl}. 
 Assume that for any $m\geq 2(l-1)^2+(l-1)-2$, if $P_j \in \mathcal I_{m,l-1},$ $1\leq j\leq l-1$, is a nonzero homogeneous polynomial 
 of degree $2(m-(l-1)+j)$ then $P_j$ contains a monomial ${\bf x}^{\bf n}$ (${\bf n}$ depending on $k$), with nonzero coefficient, that is divisible by $x_k,$ where $k=j-1,\ldots,l-1$. Consider the graded algebra homomorphism 
$f:\Q[x_1,\ldots,x_l] \rightarrow \Q[x_1,\ldots,x_{l-1}]$ defined by 
$f(x_i)=x_i$ for $i=1,\ldots,l-1$ and $f(x_l)=0$. This induces a graded algebra homomorphism
$$\bar f:\frac{\Q[x_1,\ldots,x_l]}{(R_1,\ldots,R_l)} = H^*(G_{n,l};\Q) \longrightarrow H^*(G_{n-1,l-1};\Q)= \frac{\Q[x_1,\ldots,x_{l-1}]}{(R_1^\prime,\ldots,R_{l-1}^\prime)}.$$
(This $\bar f=\hat{j}^*$ as mentioned in Fact(2) concerning the cohomology algebra of Grassmann manifold.)   
By Corollary \ref{existence of xl}, $P_l$ contains a monomial ${\bf x}^{\bf n}$, with nonzero coefficient, that is divisible by $x_{l-1}x_l$. Hence we may assume that $i<l$. 
Then $f(P_i) \in (R_1^\prime,\ldots,R_{l-1}^\prime)$ is a homogeneous polynomial of degree 
 $2(n-1-(l-1)+i)\leq 2(n-1)$. We claim that $f(P_i)\neq0$. Suppose that $f(P_i)=0$. Since $0\neq P_i\in \mathcal I_{n,l}$ is of degree $2(n-l+i)$, $P_i=Q_1R_1+\cdots+Q_iR_i$ for some homogeneous  polynomials $Q_1,\dots,Q_i \in \Q[x_1,\dots,x_l]$ with at least one $Q_r\neq 0.$ Since $\degree Q_j = 2(i-j)<2l $ for all $1 \leq j \leq i,$ $Q_1,\dots,Q_i \in \Q[x_1,\dots,x_{l-1}].$ 
 Let $s=\max\{j:Q_j \neq 0\}.$ Applying $f$ on $P_i$, we get $0=f(P_i)=f(Q_1)f(R_1)+\cdots+f(Q_i)f(R_i)=
 Q_1R_1^\prime+\cdots+Q_sR_s^\prime.$ Since $Q_s\neq 0,$ $Q_s \notin 
 (R_1^\prime,\ldots,R_{s-1}^\prime)$ (due to degree reason). This contradicts the fact that $R_1^\prime,\ldots,R_s^\prime$ form a regular sequence in $\Q[x_1,\dots,x_{l-1}].$ Hence $f(P_i)\neq0$.\\
 Since $n \geq 2l^2+l-2$, $n-1 \geq  2(l-1)^2+(l-1)-2$. Hence by induction hypothesis for each $k=i-1,\ldots,l-1$, $f(P_i)$ contains a monomial ${\bf x}^{\bf n}$ (${\bf n}$ depending on $k$), with nonzero coefficient, that is divisible by $x_k.$ Thus $P_i$ contains a monomial ${\bf x}^{\bf n}$, with nonzero coefficient, that is divisible by $x_k$, for $k=i-1,\ldots,l-1.$ 
Therefore using Corollary \ref{existence of xl}, the result follows.
 \end{proof}
 
 Next we prove the main theorem of this section.
 \begin{theorem} \label{1st theorem}
Let $1\leq k<l$ and $\bar\phi:H^*(G_{n,k};\Q)\to H^*(G_{n,l};\Q)$ be a graded algebra homomorphism.  Then $\bar\phi$ is trivial if $n\geq 2l^2+l-2$. 
\end{theorem}
\begin{proof}
Let $I=\mathcal I_{n,k}=(R_1,\ldots,R_k), I^\prime=\mathcal I_{n,l}=(R_1^\prime,\ldots,R_l^\prime)$, $H^*(G_{n,k};\Q)=\frac{\Q[x_1,\ldots,x_k]}{I}$ and 
$H^*(G_{n,l};\Q)=\frac{\Q[y_1,\ldots,y_l]}{I^\prime}$. 
Suppose $\bar \phi(x_r+I)=Q_r+I^\prime$, $r=1,\ldots,k,$ for some homogeneous polynomials $Q_r\in \Q[y_1,\ldots,y_l] $ of degree $2r.$ 
Let $\phi:\Q[x_1,\ldots,x_k]\to \Q[y_1,\ldots,y_l]$ be a graded algebra homomorphism defined by $\phi(x_r)=Q_r$, for $1\leq r\leq k$. 
Let $P_i=\phi(R_i).$  Therefore $P_i \in I^\prime $ is a homogeneous polynomial of degree $2(n-k+i)$. 
Since $\phi$ is a graded algebra homomorphism, $\max\{j:y_j \mbox{ occurs in }P_i\}=k$. 
Hence, by Corollary \ref{existence of xl}, $P_i=0$. Thus we get a graded algebra homomorphism  
$\phi:\Q[x_1,\ldots,x_k]\to\Q[y_1,\ldots,y_l]$ such that $\phi(R_i)=0$ for all $i=1,\ldots,k$. 
Hence by Lemma \ref{vanishing of Ri}, $\phi$ is trivial and thus $\bar\phi$ is trivial. 
\end{proof}

Since $H^*(G_{m-1,l};\Q)$ is a quotient of $H^*(G_{m,l};\Q)$, we obtain the following straightforward application of Theorem \ref{1st theorem}. An analogous result was obtained in \cite [Theorem 1.1]{cs} when $m-l>n-k$ and $m-l\geq 2k^2-k-1.$
\begin{corollary} \label{G_{m,l}}
Let $1\leq k<l$, $n\leq m$ and $\bar\phi:H^*(G_{n,k};\Q)\to H^*(G_{m,l};\Q)$ be a graded algebra homomorphism.  Then $\bar\phi$ is trivial if $n\geq 2l^2+l-2$.  
\end{corollary}
\noindent

We give an example to show that the conclusion of Corollary \ref{existence of xl} need not be satisfied for 
smaller values of $n,$ while Theorem \ref{1st theorem} is satisfied. 

\begin{example}
 Let $n=6, l=3$. Then 
 $$x_1^6-3x_1^4x_2+3x_1^2x_2^2-2x_2^3 = (3x_1^2-2x_2)R_1+2x_1R_2 \in \mathcal I_{6,3}$$ 
 is a homogeneous polynomial of degree $12$ which does not contain $x_3$.  \\
 \noindent Since $x_1^n \in \mathcal I_{n,1}$ and $x_1^{k(n-k)} \notin \mathcal I_{n,k}$ for $k\geq2,$ any graded algebra homomorphism from $H^*(G_{n,1};\Q) $ to $H^*(G_{n,k};\Q)$ is trivial if $k\geq2.$ In particular, any graded algebra homomorphism from $ H^*(G_{6,1};\Q)$ to $H^*(G_{6,3};\Q)$ is trivial. Also, from direct calculations, it follows that any graded algebra homomorphism from $H^*(G_{6,2};\Q)$ to $H^*(G_{6,3};\Q)$ is trivial. 
\end{example}
 
\section{Graded algebra homomorphisms from $H^*(G_{n,k};\Q)$ to $H^*(G_{n,l};\Q)$ for $k>l$} \label{section4}
In this section we consider the possible graded algebra homomorphisms from $H^*(G_{n,k};\Q)$ to $H^*(G_{n,l};\Q)$ where $k>l.$  
Unlike the case $k<l,$ Example \ref{nontrivial map} illustrates that there exist nontrivial graded algebra homomorphisms from $H^*(G_{n,k};\Q)$ to $H^*(G_{n,l};\Q)$ if $k>l$, $l=1$ and $k$ divides $n$. For $k>l,$ in certain cases, we prove that the only graded algebra homomorphism from $H^*(G_{n,k};\Q)$ to $H^*(G_{n,l};\Q)$ is trivial 
(Theorems \ref{2nd theorem part1} and \ref{2nd theorem part2}). However, the problem of determining the possible graded algebra homomorphisms  from $H^*(G_{n,k};\Q)$ to $H^*(G_{n,l};\Q)$ remains open for arbitrary $k>l.$ 

\begin{example} \label{nontrivial map}
Let $k>1$. For $n=qk$, consider the graded algebra homomorphism 
$\phi:\Q[x_1,\ldots,x_k] \rightarrow \Q[y]$ defined as 
$\phi(x_i)=0$ for $i<k$ and $\phi(x_k)=cy^k$ for some $0 \neq c \in \Q.$ Suppose $\phi(R_i)  \neq 0$ for some $i$, $1\leq i\leq k.$ Note that 
$\phi(R_i)\neq 0$ if and only if $R_i$ contains 
a monomial of the form $x_k^m$, for some $1 \leq m\leq q,$ with nonzero 
coefficient. Thus $\deg R_i=2(n-k+i)=2km$. 
Hence $i=k(m-q+1)\leq0$ unless $m=q.$ Thus $i=k.$ Also, $\phi(R_k)=(-1)^qc^qy^n.$ 
Therefore $\phi$ induces a nontrivial map $\bar{\phi}$
\begin{eqnarray*}
 \bar{\phi}: \frac{\Q[x_1,\ldots,x_k]}{(R_1,\ldots,R_k)}=H^*(G_{n,k};\Q) \rightarrow 
 H^*(G_{n,1};\Q)=\frac{\Q[y]}{(y^n)}.
\end{eqnarray*}
\end{example}

Now we provide sufficient conditions which will assure the triviality of a graded 
algebra homomorphism from $H^*(G_{n,k};\Q) $ to $H^*(G_{n,l};\Q)$ for $1<l < k.$  
We set $ H^*(G_{n,k};\Q)=\frac{\Q[x_1,\ldots,x_k]}{(R_1,\ldots,R_k)},$ 
$ H^*(G_{n,l};\Q)=\frac{\Q[y_1,\ldots,y_l]}{(R_1^\prime,\ldots,R_l^\prime)},$ 
$I=\mathcal I_{n,k}=(R_1,\ldots,R_k)$ and $I^\prime=\mathcal I_{n,l}=(R_1^\prime,\ldots,R_l^\prime).$

\begin{theorem} \label{vanishing of y_l-1 and y_l}
Let $1<l<k$ and $\phi: \Q[x_1,\ldots,x_k] \rightarrow \Q[y_1,\ldots,y_l]  $ be a graded 
algebra homomorphism such that $\phi(\mathcal I_{n,k}) \subseteq \mathcal I_{n,l}$. 
Suppose $(\phi(x_1),\ldots,\phi(x_k)) \subseteq (y_1,\ldots,y_{l-2}).$ 
Then $\phi$ 
is trivial if $n \geq 2l^2+kl-2.$
\end{theorem}
\begin{proof} 
By Lemma \ref{vanishing of Ri}, 
it suffices to show that $ \phi(R_i)=0$ for all $i=1,\ldots,k.$ 
Let $\phi(R_i)=P_i $ for some $P_i\in I^\prime.$  
Since $\deg P_i<2(n-l+1)$ for $i=1,\ldots,k-l,$ 
$P_i=0$ for $i=1,\ldots,k-l.$ Suppose $P_{k-l+i} \neq 0$ for some $i\geq 1.$  
Since $n-l+1 \geq  l(l-1)+kl+(l+1)(l-1)$,  
by Proposition \ref{existence of monomials}\eqref{existence of monomials part2}, $P_{k-l+i}$ contains a monomial $y_{l-1}^{m_1}y_l^{m_2}$ or 
 ${\bf y}^{\bf s}y_{l-1}^{m_1}y_l^{m_2}$, where $m_1\geq1$, $m_2 \geq k$ and $\wt {\bf s}<i\leq l,$ 
 with nonzero coefficient. Since 
 $\phi(x_i)$ does not contain a monomial of the form $y_{l-1}^{m_1}y_l^{m_2}$ with 
 nonzero coefficient, $P_{k-l+i}$ contains a monomial ${\bf y}^{\bf s}y_{l-1}^{m_1}y_l^{m_2}$ 
 with $m_1\geq1$, $m_2 \geq k$ and $\wt {\bf s}<l.$ 
 We have 
 \begin{eqnarray}\label{equation for R(k-l+i)} P_{k-l+i}=\phi(R_{k-l+i})=\sum_{\substack{{\bf n} \in \N^k, \wt {\bf n}=n-l+i,\\n_{l+1}+\cdots+n_k <l}}\phi(c_{\bf n}{\bf x}^{{\bf n}})+
  \sum_{\substack{{\bf n} \in \N^k, \wt {\bf n}=n-l+i,\\n_{l+1}+\cdots+n_k \geq l}}\phi(c_{\bf n}{\bf x}^{{\bf n}}).
  \end{eqnarray}
  If $n_{l+1}+\cdots+n_k \geq l$ then 
for every monomial ${\bf y}^{{\bf s}^\prime}y_{l-1}^{m_3}y_l^{m_4}$ in $\phi({\bf x}^{\bf n})$, 
$\wt {\bf s}^\prime \geq l. $ 
Hence ${\bf y}^{\bf s}y_{l-1}^{m_1}y_l^{m_2}$ does not occur in the second summand of (\ref{equation for R(k-l+i)}).  
Suppose $n_{l+1}+\cdots+n_k <l.$ Since the highest possible 
power of $y_l$ in a monomial in $ \phi(x_j)$ is $\lfloor \frac{k}{l}\rfloor$ for all $j=1,\ldots,k$ and $\phi(x_l)$ does not contain $y_l$ with nonzero coefficient, the power of $y_l$ in a monomial in $\phi({\bf x}^{\bf n})$ is $\leq(n_{l+1}+\cdots+n_k)
\lfloor\frac{k}{l} \rfloor<l\lfloor\frac{k}{l}\rfloor \leq k$. This implies that the power of $y_l$ in a monomial in $\phi({\bf x}^{\bf n})<k$.    
Hence ${\bf y}^{\bf s}y_{l-1}^{m_1}y_l^{m_2}$ does not occur in the first summand of (\ref{equation for R(k-l+i)}). 
This implies that $\phi(R_{k-l+i})$ does not contain ${\bf y}^{\bf s}y_{l-1}^{m_1}y_l^{m_2}$, which leads to a contradiction. 
\end{proof}

\begin{corollary} \label{corollary vanishing of y_l-1 and y_l}
 Let $1<l<k$ and $\bar{\phi}: H^*(G_{n,k};\Q) \rightarrow H^*(G_{n,l};\Q)  $ be a graded 
algebra homomorphism. 
For $i=1,\ldots,k,$ let $\bar\phi(x_i+\mathcal I_{n,k})=Q_i+\mathcal I_{n,l}$ for some homogeneous polynomials $Q_i \in \Q[y_1,\ldots,y_l].$
Suppose 
$(Q_1,\ldots,Q_k) \subseteq (y_1,\ldots,y_{l-2}).$ 
Then $\bar \phi$ 
is trivial if $n \geq 2l^2+kl-2.$
\end{corollary}
\begin{proof}
 Let $\phi:\Q[x_1,\ldots,x_k] \rightarrow \Q[y_1,\ldots,y_l]$ be a graded algebra homomorphism defined as 
$\phi(x_i)=Q_i.$ Since $\phi$ induces a map $\bar\phi,$ $\phi(I)\subseteq I^\prime.$ 
Therefore by Theorem \ref{vanishing of y_l-1 and y_l}, the result follows.
\end{proof}

The following theorem generalizes Theorem \ref{vanishing of y_l-1 and y_l}.

\begin{theorem} \label{vanishing of y_l}
 Let $1<l<k$ and $\phi: \Q[x_1,\ldots,x_k] \rightarrow \Q[y_1,\ldots,y_l] $ be a graded 
algebra homomorphism such that $\phi(\mathcal I_{n,k}) \subseteq \mathcal I_{n,l}$. 
Suppose $(\phi(x_1),\ldots,\phi(x_k)) \subseteq (y_1,\ldots,y_{l-1}).$ 
Then $\phi$ is trivial if $n \geq 3k^2-2.$
\end{theorem}
\begin{proof}
 Let $\psi:\Q[y_1,\ldots,y_l] \rightarrow \Q[y_{l-1},y_l]$ be a graded algebra homomorphism defined as 
 $\psi(y_i)=0$ for $i=1,\ldots,l-2, \psi(y_{l-1})=y_{l-1}$ and $\psi(y_l)=y_l.$ 
 Let $J=(\psi(R_1^\prime),\ldots,\psi(R_l^\prime))$ and 
 $h=\psi \circ \phi.$  As $n\geq3k^2-2,$ so $n \geq 2l^2+kl-2.$ 
 Hence
 by Theorem \ref{vanishing of y_l-1 and y_l}, it suffices to show that $h$ is trivial. 
 We show that $ h(R_i)=0 $ for all $i=1,\ldots,k.$ 
 Since $\deg R_i<2(n-l+1)$ for $i=1,\ldots,k-l$ and $\phi(R_i) \in I^\prime,$ 
 $\phi(R_i)=0$ for $1 \leq i \leq k-l$ and hence $h(R_i)=0$ for $1 \leq i \leq k-l. $ 
 Let $\psi(R_i^\prime)=Q_i$ for $i=1,\ldots,l.$ Since $\phi(R_i) \in I^\prime,$ 
 $h(R_i) \in J$ for all $1 \leq i \leq k.$ 
 Therefore 
 \begin{eqnarray}
\label{equation for lambdai}h(R_{k-l+i}) &=& \lambda_i Q_i \mbox{ for }1 \leq i<l \mbox{ and }\\
\label{equation for lmk}h(R_k)&=& \lambda_l Q_l + \lambda_l^\prime y_{l-1} Q_1.
\end{eqnarray} 
 We claim that $Q_i$ contains a monomial 
 $y_{l-1}^ry_l^s$ with $1\leq r\leq l,$ $s\geq l.$ Since $n-l+i \geq n-l+1 \geq n-k+1 \geq 3k^2-k-1\geq 2k^2-1\geq l(l-1)+l^2+(l-1),$ 
 by Proposition \ref{proposition of gh}, there exist integers $r\geq 1, s \geq l,$ such that $n-l+i=r(l-1)+sl.$ 
 Hence $Q_i$ contains a monomial $y_{l-1}^ry_l^s.$ 
 If $r > l, $ then replace $r$ by $r_1=r-l$ and $s$ by $s_1=s+l-1.$ If 
 $r_1 \leq l$ then we are done. Otherwise continue as above. This proves the claim. \\
Note that if ${\bf x}^{\bf n}$ occurs in $R_{k-l+i},$ $1\leq i\leq l,$ then $|{\bf n}| \geq 2k.$ Suppose $|{\bf n}|<2k.$ 
Then $\deg {\bf x}^{\bf n}=2n_1+\cdots+2kn_k\leq 
2k(n_1+\cdots+n_k)<2k(2k).$ But $\deg R_{k-l+i} \geq 2(n-l+1) 
\geq 2(3k^2-k-1)\geq 2(2k^2),$ a contradiction. Hence $|{\bf n}| \geq 2k.$ 
Since $\phi(x_j)$ does not contain a pure monomial in $y_l$ and $\psi(y_j)=0$ for $1 \leq j<l-1,$ 
$h({\bf x}^{\bf n})$ is sum of monomials of the form $y_{l-1}^{m_1}y_l^{m_2}$ with $m_1 \geq 2k > k>l$, for a monomial ${\bf x}^{\bf n}$ occurring in $R_{k-l+i}.$ Hence comparing the coefficients of 
$y_{l-1}^ry_l^s$, where $1\leq r\leq l$, $s\geq l$ in Equation (\ref{equation for lambdai}) we get that $\lm_i=0.$ \\
Let $i=l.$  
If $h({\bf x}^{\bf n}) $ contains a monomial $y_{l-1}^ry_l^s$ with nonzero coefficient 
then $r \geq 2k>2l.$ 
Note that $Q_l$ contains a monomial of the form $y_{l-1}^ry_l^s,$ 
with $1\leq r\leq l$ and $s \geq l.$ Let $r_1=r+l,s_1=s-(l-1).$ Then comparing the coefficients of 
$y_{l-1}^ry_l^s$ and $y_{l-1}^{r_1}y_l^{s_1}$ in Equation (\ref{equation for lmk}), 
we get that 
\begin{eqnarray}
 \label{equation1 for lmk in case2}
  \frac{(r+s)!}{r!s!}\lm_l-\frac{(r+s-1)!}{(r-1)!s!} \lm_l^\prime&=&0\\
   \label{equation2 for lmk in case2}
  \frac{(r_1+s_1)!}{r_1!s_1!}\lm_l-\frac{(r_1+s_1-1)!}{(r_1-1)!s_1!} \lm_l^\prime&=&0.
\end{eqnarray}
Solving Equations (\ref{equation1 for lmk in case2}) and (\ref{equation2 for lmk in case2}), 
we get that $\lm_l=\lm_l^\prime=0.$ Hence $h(R_k)=0.$ Therefore by Lemma 
\ref{vanishing of Ri}, $h$ is trivial. 
\end{proof}

\begin{corollary}\label{cor}
 Let $\bar{\phi}: H^*(G_{n,k};\Q) \rightarrow H^*(G_{n,l};\Q)  $, $1<l<k$ be a graded 
algebra homomorphism. 
For $i=1,\ldots,k,$ let $\bar\phi(x_i+\mathcal I_{n,k})=Q_i+\mathcal I_{n,l}$ for some homogeneous polynomials $Q_i \in \Q[y_1,\ldots,y_l].$ Suppose 
$(Q_1,\ldots,Q_k) \subseteq (y_1,\ldots,y_{l-1}).$ 
Then $\bar \phi$ is trivial if $n \geq 3k^2-2.$
\end{corollary}
\begin{proof}
Using argument as in Corollary \ref{corollary vanishing of y_l-1 and y_l}, the result follows from Theorem \ref{vanishing of y_l}. 
\end{proof}

As a consequence we determine the possible graded algebra homomorphisms from $H^*(G_{n,k};\Q)$ to $H^*(G_{n,l};\Q)$ when $2<l<k<2(l-1).$
\begin{theorem}\label{2nd theorem part1}
 Let $2<l<k<2(l-1)$ and $\bar\phi:H^*(G_{n,k};\Q)\to H^*(G_{n,l};\Q)$ be a graded algebra homomorphism. Then $\bar\phi$ is trivial if $n\geq3k^2-2.$ 
\end{theorem}
\begin{proof}
For $i=1,\ldots,k,$ let $\bar{\phi}(x_i+I)=P_i+I^\prime$ with $\deg P_i=\deg x_i.$  
Let $\phi:\Q[x_1,\ldots,x_k] \rightarrow \Q[y_1,\ldots,y_l]$ be a graded algebra homomorphism defined as 
$\phi(x_i)=P_i $. By Theorem \ref{vanishing of y_l}, it suffices to show that $(P_1,\ldots,P_k) \subseteq (y_1,\ldots,y_{l-1}).$ Let $\psi:\Q[y_1,\ldots,y_{l}] \rightarrow \Q[y_{l-1},y_l]$ be a graded algebra homomorphism defined as $\psi(y_i)=0$ for $i=1,\ldots,l-2$ and $\psi(y_{i})=y_{i}$ for $i=l-1,l.$ 
Let $h=\psi \circ \phi.$ Since $k<2(l-1)$, $P_i$ does not contain monomials of the form $y_{l-1}^{m_1}y_l^{m_2}$ with $m_1+m_2 > 1.$  Hence $h(x_i)=0$ for $i=1,\ldots,l-2,l+1,\ldots,k.$ Let $h(x_{l-1})=a y_{l-1}$ and $h(x_l)=by_l$ for some $a,b \in \Q.$ Since 
$\deg R_i<2(n-l+1)$ for $i=1,\ldots, k-l$ and $h(R_i)=\psi(\phi(R_i))\in (\psi(R_1^\prime),\ldots,\psi(R_l^\prime)), $ $h(R_i)=0$ for $i=1,\ldots,k-l.$ 
Since $n-k+1 \geq k(k-1)+(k-1)+k \geq l(l-1)+(l-1)+l$, by Proposition \ref{proposition of gh}, there exist integers $r_1\geq 1$ and $s_1 \geq 1 $ 
such that $n-k+1=r_1(l-1)+s_1l. $ Hence comparing the coefficients 
of $y_{l-1}^{r_1}y_l^{s_1}$ in $h(R_1),$ we get $a^{r_1}b^{s_1}=0.$\\
Suppose $b\neq 0.$ Then $a=0.$ This implies
\[
 h(R_i)=\begin{cases}
         c_iy_l^{m_i}\mbox{ for some }c_i\in \Q\mbox{ and }m_i \in \N&\hspace{0.5in}\mbox{if }l\mbox{ divides }(n-k+i)\\
         0& \hspace{0.5in}\mbox{otherwise}. 
        \end{cases}
\]
Let $Q_i=\psi(R_i^\prime).$ Thus 
 \begin{eqnarray}\label{equation for lmi}
 \label{equation for lmi in case}h(R_{k-l+i}) &=& \lambda_i Q_i \mbox{ for }1 \leq i<l \mbox{ and }\\
\label{equation for lmk in case}h(R_k)&=& \lambda_l Q_l + \lambda_l^\prime y_{l-1} Q_1.
\end{eqnarray} 
Since $n-l+i  \geq l(l-1)+(l-1)+l$, by Proposition \ref{proposition of gh}, there exist integers $r_2\geq 1$ and $s_2 \geq 1 $ 
such that $n-l+i=r_2(l-1)+s_2l. $ Hence comparing the coefficients 
of $y_{l-1}^{r_2}y_l^{s_2}$ in Equation (\ref{equation for lmi in case}), we get 
$\lm_i=0$ for $1 \leq i<l.$ Thus $h(R_{k-l+i})=0$ for $1 \leq i<l.$ 
Let $h(R_k)=cy_l^s$. Then $n=sl=l(l-1)+(s-l+1)l=2l(l-1)+(s-2l+2)l.$  
 Since $sl=n \geq 2l^2-1 \geq 2l^2-2l,$ $s \geq 2l-2.$ 
Thus $s-2l+2\geq0$ which implies that $s-l+1\geq0.$  Hence comparing the coefficients of 
$y_{l-1}^ly_l^{s-l+1}$ and $y_{l-1}^{2l}y_l^{s-2l+2}$ in Equation (\ref{equation for lmk in case}), 
we get 
\begin{eqnarray} 
\label{equation1 for lmi in case} 
\frac{(s+1)!}{l!(s-l+1)!}\lm_l - \frac{s!}{(l-1)!(s-l+1)!}\lm_l^\prime&=&0\\
\label{equation2 for lmi in case} 
\frac{(s+2)!}{(2l)!(s-2l+2)!}\lm_l -\frac{(s+1)!}{(2l-1)!(s-2l+2)!}\lm_l^\prime&=&0.
\end{eqnarray}
Solving Equation (\ref{equation1 for lmi in case}) and (\ref{equation2 for lmi in case}) we get $\lm_l=\lm_l^\prime=0.$ Hence $h(R_k)=0.$ 
Note that $x_l^{s_j}$ occurs with nonzero coefficient in $R_{k-l+j}$ for some $1 \leq j\leq l,$ where
$s_j$ is a positive integer.
Hence comparing the coefficients of $y_l^{s_j}$ in $h(R_{k-l+j}),$ we get 
$b=0,$ a contradiction. Hence $b=0.$
\end{proof}

Now, as an application of Corollary \ref{cor}, we shall attain another sufficient condition for the triviality of a graded algebra homomorphism from $H^*(G_{n,k};\Q)$ to $H^*(G_{n,l};\Q)$.

\begin{discussionbox}
\label{discussion}
Let $\rho:\Q[y_1,\ldots,y_l]\to\Q[y_l]$ be a graded algebra homomorphism defined by $\rho(y_i)=0$ for $i=1,\ldots,l-1$ and $\rho(y_l)=y_l.$ Let $T_i=\rho(R_i^\prime)$. Then $\rho$ induces a graded algebra homomorphism 
\begin{eqnarray}
 \bar\rho:\frac{\Q[y_1,\ldots,y_l]}{(R_1^\prime,\ldots,R_l^\prime)}\to\frac{\Q[y_l]}{(T_1,\ldots,T_l)}.
\end{eqnarray}
Let $\bar\phi:H^*(G_{n,k};\Q)\to H^*(G_{n,l};\Q),$ $n \geq 3k^2-2$, be a graded algebra homomorphism defined by  $\bar\phi(x_i+I)=P_i+I^\prime$, where $\deg P_i=\deg x_i.$ Therefore, we get the following graded algebra homomorphism
\begin{eqnarray}\label{solution}
 \bar\rho\circ\bar\phi:\frac{\Q[x_1,\ldots,x_k]}{(R_1,\ldots,R_k)}\to\frac{\Q[y_l]}{(T_1,\ldots,T_l)}.
\end{eqnarray}
Therefore, in order to show that $\bar \phi$ is trivial, by Corollary \ref{cor}, it suffices to show that $(P_1,\ldots,P_k)\subset(y_1,\ldots,y_{l-1})$ 
which is equivalent to the triviality of the map $\bar\rho\circ\bar\phi.$ \\
Let $k=el+f$ and $n=e_1l+f_1$ for some integers $e,e_1,f,f_1$, where $0\leq f,f_1<l.$ 
Note that, if $j$ is not a multiple of $l$ then $\bar\rho\circ\bar\phi(x_j+\mathcal I_{n,k})=0.$ Let $\bar\rho\circ\bar\phi(x_{il}+\mathcal I_{n,k})=\tau_iy_l^i+(T_1,\ldots,T_l)$ for $i=1,\ldots,e.$ Also, 
\[
 T_i=\begin{cases}
  (-1)^{e_1} y_l^{e_1}& \mbox{ if }i=l-f_1\\
  0 & \mbox{ otherwise }.
 \end{cases}
\]
Let $R_{t_1},\ldots,R_{t_s}$ ($s=e$ or $e+1$) be those $R_i$'s which have degrees that are multiples of $2l.$ 
Let $\deg R_{t_j}=2lq_j$, for $j=1,\ldots,s,$ for some integers $q_j.$ Then for $j=1,\ldots,s,$
\[
 q_j=\begin{cases}
     e_1-e+j & \mbox{ if }s=e\\
     e_1-e+(j-1) & \mbox{ if }s=e+1.
     \end{cases}
\]
\noindent
Let $S_j=\bar \rho \circ \bar \phi(R_{t_j}).$ Then
\begin{eqnarray*}\label{equation for Sj}
 S_j=\left(\underset{{\bf n}\in\mathbb{N}^e,{\wt {\bf n}}=q_j}{\sum}(-1)^{|{\bf n}|}c_{\bf n}\tau^{\bf n}\right)y_l^{q_j}+(T_1,\ldots,T_l).
\end{eqnarray*}
Let $S_j^\prime:=\underset{{\bf n}\in\mathbb{N}^e,{\wt {\bf n}}=q_j}{\sum}(-1)^{|{\bf n}|}c_{\bf n}\tau^{\bf n}.$
Since $S_j=0$ for $j\neq s,$ we get 
\begin{eqnarray}\label{eqn}
 S_j^\prime=0 \mbox{ for }j\neq s.
\end{eqnarray}
Define a graded algebra homomorphism $\theta:\Q[x_1,\ldots,x_e]\to\Q[y]$  by $\theta(x_i)=\tau_iy^i$, where $\deg y=2.$
Let $R_j^\prime$ be the part of degree $2(e_1-e+j)$ of the formal inverse of $c=1+x_1+\cdots+x_e.$ Thus, $\mathcal{I}_{e_1,e}=(R_1^\prime,\ldots,R_e^\prime).$ Now, if $s=e,$ then 
$$\theta(R_j^\prime)=\theta\left(\underset{{\bf n}\in\mathbb{N}^e,{\wt {\bf n}}=q_j}{\sum}(-1)^{|{\bf n}|}c_{\bf n}{\bf x}^{\bf n}\right)=\left(\underset{{\bf n}\in\mathbb{N}^e,{\wt {\bf n}}=q_j}{\sum}(-1)^{|{\bf n}|}c_{\bf n}\tau^{\bf n}\right)y^{q_j}=S_j^\prime y^{q_j}=0$$
for $j=1,\ldots,e-1.$ Then $\theta$ induces a graded algebra homomorphism from $H^*(G_{e_1,e};\Q)$ to $H^*(G_{e_1,1};\Q).$ Therefore, if any graded algebra homomorphism from $H^*(G_{e_1,e};\Q)$ to $H^*(G_{e_1,1};\Q)$ 
is trivial, then the homomorphism in (\ref{solution}) is trivial. 
\end{discussionbox}

For $s=e+1,$ we obtain the following theorem. 
\begin{theorem}\label{2nd theorem part2}
 Let $1<l<k,$ $k=el+f$ and $n=e_1l+f_1$ for some integers $e,f,e_1,f_1$ such that $0\leq f,f_1<l.$ Let  $\bar\phi:H^*(G_{n,k};\Q)\to H^*(G_{n,l};\Q)$ be a graded algebra homomorphism. Then 
$\bar\phi$ is trivial if $f> f_1 $ and $ n\geq 3k^2-2.$  
\end{theorem}
\begin{proof}
Let the notations be as in Discussion \ref{discussion}. Since $f>f_1,$ $s=e+1.$ Hence, by Equation (\ref{eqn}),  
$S_j^\prime=0$ for $j=1,\ldots,e.$
Define a graded algebra homomorphism $\theta:\Q[x_1,\ldots,x_e]\to\Q[y]$ by $\theta(x_i)=\tau_iy^i$, where $\deg y=2.$ Let $\mathcal{I}_{e_1-1,e}=(R_1^{\prime\prime},\ldots,R_{e}^{\prime\prime}),$ where $R_j^{\prime\prime}$ is the part of degree $2(e_1-1-e+j)$ of the formal inverse of $c=1+x_1+\cdots+x_e.$ Then
$$ 
\theta(R_j^{\prime \prime})=S_j^\prime y^{e_1-1-e+j}=0
$$
for $j=1,\ldots,e.$
Thus by Lemma  \ref{vanishing of Ri}, $\theta$ is trivial, which implies that $\tau_i=0$ for all $i$. Hence 
$\bar\rho\circ\bar\phi$ is trivial. Therefore, by Corollary \ref{cor}, $\bar\phi$ is trivial. 
\end{proof}

\begin{remark}
Note that if $f\leq f_1$, then the number of $R_i$'s having degrees multiple of $2l$ is $e$, i.e. $s=e.$ 
Thus in order to show that $\bar \phi$ is trivial, from Discussion \ref{discussion}, it suffices to show that every graded algebra homomorphism from $H^*(G_{e_1,e};\Q)$ to $H^*(G_{e_1,1};\Q)$ is trivial.
But, from the existence of a nontrivial graded algebra homomorphism from $H^*(G_{e_1,e};\Q)$ to $H^*(G_{e_1,1};\Q)$, we cannot conclude that $\bar \phi$ is trivial.

\end{remark}

\section{ Proof of Theorem \ref{main theorem}}
At first, we recall a relation between the homotopy 
class of a map and the homomorphism it induces in cohomology with rational coefficients.  Familiarity with basic notions in rational homotopy theory has been assumed.  For further details, see \cite{FHT,griffiths-morgan}.\\
  Let $X$ be any simply connected finite CW complex and let $X_0$ denote its rationalization.   Denoting the minimal model of $X$ by 
$\mathcal{M}_X$, one has a bijection $[X_0,Y_0]\cong [\mathcal{M}_Y,\mathcal{M}_X], [h]\mapsto [\Phi_h],$ 
where on the left we have homotopy classes of continuous 
maps $X_0\to Y_0$ and on the right we have homotopy classes of differential graded commutative algebra
 homomorphisms of the minimal models $\mathcal{M}_Y\to \mathcal{M}_X.$
 In the case when $X=U(n)/(U(n_1)\ldots\times U(n_r))$ is a complex flag manifold (i.e. when $n=\sum_{i=1}^r{n_i}$), then $X$ is a K\"ahler manifold and hence is formal, that is there is a morphism of differential graded commutative algebras $\rho_X:\mathcal{M}_X\to H^*(X;\Q)$ which induces an isomorphism in cohomology, where $H^*(X;\Q)$ is endowed with the zero differential.  Moreover, it is known that when both $X$ and $Y$ are complex flag manifolds, any continuous map $h:X\to Y$ is formal, that is, the homotopy class of the morphism $h_0:X_0\to Y_0$ is determined by the graded $\Q$-algebra homomorphism $h^*:H^*(Y;\Q)\to H^*(X;\Q).$  The above discussion can be written as the following theorem.

\begin{theorem}\label{rigid}(\cite{glover-homer2}, Theorem 1.1)
Let $X,Y$ be complex flag manifolds.  Then $[h]\mapsto H^*(h;\mathbb{Q})$ establishes a bijection from $[X_0,Y_0]$ to the set of graded $\mathbb{Q}$-algebra 
homomorphisms   $Hom_{alg}(H^*(Y;\mathbb{Q}),H^*(X;\mathbb{Q}))$. 
\end{theorem}

We now turn to the proof of Theorem \ref{main theorem}.

\begin{proof}[Proof of Theorem \ref{main theorem}]
Let $h:G_{n,l} \to G_{n,k}$ be a continuous map. By Theorem \ref{1st theorem} and Theorem \ref{2nd theorem part1}, $h^*$ is trivial in Cases $(i),(ii)(a)$. In Case $(ii)(b)$, $h^*$ satisfies the hypothesis of Theorem \ref{2nd theorem part2} with $f>0$, $f_1=0$. So, $h^*$ is trivial in Case $(ii)(b).$ Then by Theorem \ref{rigid}, $h_0$ is null-homotopic in all three cases.
\end{proof}
\begin{remark}\label{other result}
$(1)$ The above proof also establishes the following statement, which is a generalization of Case $(ii)(b)$ of Theorem \ref{main theorem}.\\
{\it Any continuous map $\phi:G_{n,l}\to G_{n,k}$ is rationally null homotopic if $1<l<k$, $f> f_1, n\geq 3k^2-2,$ where $k=el+f$ and $n=e_1l+f_1$ for some integers  $e,f,e_1,f_1$ such that $0\leq f,f_1<l.$}\\ 
$(2)$ The set of homotopy classes of continuous maps $[G_{n,l},G_{n,k}]$ is finite when $n,k,l$ are as in Theorem \ref{main theorem} and Remark \ref{other result}$(1)$. 
This follows from the finiteness of the set of homotopy classes of continuous maps $h:X\to Y$ having the same rationalization $h_0:X_0\to Y_0$.  See \cite[\S 12]{sullivan}. 
\end{remark}

As an application of Theorem \ref{main theorem}, we get the following {\it invariant subspace theorem}. If $y\in G_{n,k}$ and $x\in G_{n,l}$ and $k<l$, then it makes sense to write $y\subset x$.
\begin{theorem}\label{subspace invariant theorem}
Let $h:G_{n,l}\to G_{n,k}$ be any continuous map where $n,k,l$ are as in Theorem \ref{main theorem}$(i)$. Then there exists an element $x\in G_{n,l}$ such that $h(x)\subset x.$
\end{theorem}
\begin{proof}
 The proof of this theorem is analogous to that of \cite[Theorem 1.3]{cs} with obvious changes.
\end{proof}

\section{Examples}
In this section we provide a few examples, which do not follow from Theorems \ref{1st theorem}, \ref{2nd theorem part1}, or \ref{2nd theorem part2}, where the graded algebra homomorphisms between cohomology algebras of distinct Grassmann manifolds are trivial. We also give examples of nontrivial graded algebra homomorphisms from $H^*(G_{n,2};\Q)$ to $H^*(G_{n,1};\Q)$ even if $n$ is not even (see Example \ref{nontrivial map} for $n$ even).

In the following proposition we consider possible graded algebra homomorphisms from $H^*(G_{2m,3};\Q)$ to $H^*(G_{2m,2};\Q).$ Note that using Theorem \ref{2nd theorem part2}, we can conclude that any such homomorphism is trivial provided $m\geq 13.$ By direct calculations, we show that any such homomorphism is trivial for $m\geq 3.$  

\begin{proposition}
 Let $\bar\phi:H^*(G_{2m,3};\Q) \rightarrow H^*(G_{2m,2};\Q)$ be a graded algebra homomorphism. Then 
 $\bar\phi$ is trivial if $m \geq 3.$ 
\end{proposition}
\begin{proof}
In Example \ref{lastexample} (1), we shall show that any graded algebra homomorphism from $H^*(G_{6,3};\Q)$ to $H^*(G_{6,2};\Q)$ is trivial. Hence we assume that $m \geq 4.$

\noindent
Let $I=\mathcal I_{2m,3}=(R_1,R_2,R_3), I^\prime=\mathcal I_{2m,2}=(R_1^\prime,R_2^\prime),$ $H^*(G_{2m,3};\Q)=\frac{\Q[x_1,x_2,x_3]}{I}$ and 
$H^*(G_{2m,2};\Q)=\frac{\Q[y_1,y_2]}{I^\prime}.$ 
Let $\bar\phi(x_1+I)=a_{10}y_1+I^\prime, \bar\phi(x_2+I)=a_{20}y_1^2+a_{01}y_2+I^\prime$ and 
$\bar\phi(x_3+I)=a_{30}y_1^3+a_{11}y_1y_2+I^\prime.$ 
Define $\phi:\Q[x_1,x_2,x_3] \rightarrow \Q[y_1,y_2] $ as 
$\phi(x_1)=a_{10}y_1, \phi(x_2)=a_{20}y_1^2+a_{01}y_2$ and 
$\phi(x_3)=a_{30}y_1^3+a_{11}y_1y_2.$ It suffices to show that $\phi$ 
is trivial. Since $\phi(I) \subseteq I^\prime$ and 
$\deg R_1=4m-4,$ $\phi(R_1)=0. $ Hence comparing the coefficients of $y_2^{m-1}$ in 
$\phi(R_1),$ we get $a_{01}=0.$ \\
\underline{Case 1:} $2m-2 =3q$ for some $q \in \N.$\\
In this case, $x_3^q$ occurs in $R_1$ with nonzero coefficient. Hence 
comparing the coefficients of $(y_1y_2)^q$ in $\phi(R_1),$ we get 
$a_{11}=0.$ Thus $(\phi(x_1),\phi(x_2),\phi(x_3)) \subseteq (y_1).$ Next we shall prove that 
$\phi(R_2)=\phi(R_3)=0.$ Suppose $\phi(R_2) \neq 0.$ Since $\phi(R_2) \in I^\prime,$ 
by Corollary \ref{existence of xl}, $\phi(R_2)$ contains a monomial $y_1^{m_1}y_2^{m_2}$ 
with nonzero coefficient and $m_2>0,$ which is a contradiction. Hence $\phi(R_2)=0.$ 
Similarly, $\phi(R_3)=0.$ Hence by Lemma \ref{vanishing of Ri}, $\phi$ is trivial.\\
\underline{Case 2:} $2m-2=3q+1$ for some $q \in \N.$\\
In this case, $x_1x_3^q$ occurs in $R_1$ with nonzero coefficient. Hence 
comparing the coefficients of $y_1^{q+1}y_2^q,$ we get $a_{10}a_{11}^q=0.$ If 
$a_{11}=0$ then argument as in Case $1$ gives the result. Suppose 
$a_{11}\neq0.$ Then $a_{10}=0.$ Now, comparing the coefficients of $y_1^{q+3}y_2^{q-1}$ in $\phi(R_1),$ we get 
$a_{11}^{q-1}a_{20}^2=0.$ Hence $a_{20}=0.$ Thus $\phi(x_1)=\phi(x_2)=0.$ 
Since $\deg R_2$ is not a multiple of $6$, $\phi(R_2)=0.$ 
Let 
\begin{eqnarray} \label{equation for R3}
 \phi(R_3)=\lm y_1 R_1^\prime+\lm^\prime R_2^\prime.
\end{eqnarray}
Since $\phi(R_3)=\phi((-1)^{q+1}x_3^{q+1})=(-1)^{q+1}(a_{30}y_1^3+a_{11}y_1y_2)^{q+1},$ 
comparing the coefficients of $y_2^m,$ in Equation (\ref{equation for R3}), we get 
$\lm^\prime =0.$ Then comparing the coefficients of $y_1^2y_2^{m-1},$ in Equation 
(\ref{equation for R3}), we get $\lm=0.$ Thus $\phi(R_3)=0.$ Hence by Lemma \ref{vanishing of Ri}, 
$\phi$ is trivial.\\
\underline{Case 3:} $2m-2=3q+2$ for some $q \in \N.$\\
Let 
\begin{eqnarray} \label{equation for R2}
 \phi(R_2)=\lm R_1^\prime.
 \end{eqnarray} 
Then comparing the coefficients of $y_1y_2^{m-1}$ in Equation (\ref{equation for R2}), we get $\lm=0.$ Hence $\phi(R_2)=0.$ 
Therefore comparing the coefficients of $(y_1y_2)^{q+1}$ in $\phi(R_2),$ we get 
$a_{11}=0.$ Then argument as in Case $1$ completes the proof.
\end{proof}

In Examples \ref{lastexample} $(1)$ and $(2)$, we show that all graded algebra homomorphisms from $H^*(G_{n,3};\Q)$ to $H^*(G_{n,2},\Q)$
are trivial, if $n$ is $6$ and $7$ respectively. 
In Example \ref{lastexample} $(3)$, we illustrate that the existence of a non-trivial graded algebra homomorphism 
between the cohomology algebras of distinct complex Grassmann manifolds depends upon the coefficient field. 

\begin{example}\label{lastexample}
$(1)$ Let $\bar\phi:H^*(G_{6,3};\Q) \rightarrow H^*(G_{6,2},\Q)$ be a graded algebra homomorphism. Then 
 $\bar\phi$ is trivial.\\
 {\it Proof:} Let $I=\mathcal I_{6,3}, I^\prime=\mathcal I_{6,2}.$ Let $\bar\phi(x_1+I)=a_{10}y_1+I^\prime, \bar\phi(x_2+I)=a_{20}y_1^2+a_{01}y_2+I^\prime$ and 
$\bar\phi(x_3+I)=a_{30}y_1^3+a_{11}y_1y_2+I^\prime.$ 
Define $\phi:\Q[x_1,x_2,x_3] \rightarrow \Q[y_1,y_2] $ by 
$\phi(x_1)=a_{10}y_1, \phi(x_2)=a_{20}y_1^2+a_{01}y_2$ and 
$\phi(x_3)=a_{30}y_1^3+a_{11}y_1y_2.$ It suffices to show that $\phi$ 
is trivial. Since $\phi(R_1)=\phi(x_1^4 -3 x_1^2x_2+2 x_1x_3+x_2^2)=0,$ 
comparing the coefficients of $y_2^2 ,$ we get $a_{01}=0.$ Hence comparing the coefficients 
of $y_1^2y_2$ in $\phi(R_1),$ we get $a_{10}a_{11}=0. $ Suppose $a_{11}=0.$ 
Let $\phi(R_2)=\lm R_1^\prime.$ Then comparing the coefficients of $y_1y_2^2,$ 
we get $\lm=0.$ Thus $\phi(R_2)=0.$ Let $\phi(R_3)=\lm_1 y_1 R_1^\prime+\lm_2 R_2^\prime.$ 
Hence comparing the coefficients of $y_2^3,$ we get $\lm_2=0.$ Then comparing the 
coefficients of $y_1^2y_2^2,$ we get $\lm_1=0.$ Thus $\phi(R_3)=0.$ Therefore by 
Lemma \ref{vanishing of Ri}, $\phi$ is trivial. \\
Suppose $a_{11} \neq 0.$ Then $a_{10}=0.$ Let $\phi(R_2)=\lm R_1^\prime.$ Hence 
comparing the coefficients of $y_1y_2^2$, we get $ \lm=0.$ Thus $\phi(R_2)=0.$ 
Since $\phi(x_1)=a_{10}y_1=0,$ $\phi(R_2)=\phi(2x_2x_3).$ This implies that 
$\phi(x_2)\phi(x_3)=0.$ Comparing the coefficients of $y_1^3y_2,$ we get $a_{20}=0.$ 
Thus $\phi(x_2)=0.$ Hence $\phi(R_3)=\phi(x_3^2).$ 
Let $\phi(R_3)=\lm_1 y_1 R_1^\prime+\lm_2 R_2^\prime.$ Then comparing the coefficients 
of $y_2^3,$ we get $\lm_2=0.$ Therefore $\phi(R_3)=\phi(x_3^2)= \lm_1 y_1 R_1^\prime.$ 
Thus $(a_{30}y_1^3+a_{11}y_1y_2)^2=\lm_1 y_1(-y_1^5+4y_1^3y_2-3y_1y_2^2).$ 
Comparing the coefficients and solving, we get $\lm_1=0.$ Thus $\phi(R_3)=0.$ 
Therefore by Lemma \ref{vanishing of Ri}, $\phi$ is trivial.\\
$(2)$
 Let $\bar\phi:H^*(G_{7,3};\Q) \rightarrow H^*(G_{7,2},\Q)$ be a graded algebra homomorphism. Then 
 $\bar\phi$ is trivial.\\
 {\it Proof:} Let $I=\mathcal I_{7,3},I^\prime=\mathcal I_{7,2}.$ Let $\bar\phi(x_1+I)=ay_1+I^\prime, \bar\phi(x_2+I)=by_1^2+cy_2+I^\prime$ and 
$\bar\phi(x_3+I)=dy_1^3+ey_1y_2+I^\prime.$ 
Define $\phi:\Q[x_1,x_2,x_3] \rightarrow \Q[y_1,y_2] $ by 
$\phi(x_1)=ay_1, \phi(x_2)=by_1^2+cy_2$ and 
$\phi(x_3)=dy_1^3+ey_1y_2.$ It suffices to show that $\phi$ 
is trivial. Since $\phi(R_1)=0,$ comparing the coefficients of $y_1y_2^2,$ we get 
$c(2e-3ac)=0.$ We claim that $c=0.$ Suppose $c \neq 0.$ Then $2e=3ac.$ Let $\phi(R_2)=\lm R_1^\prime.$ Comparing the coefficients 
of $y_2^3$ and $y_1^2y_2^2,$ we get $c^3=\lm$ and $6a^2c^2-3ace-3bc^2+e^2=6 \lm=6 c^3.$ 
Using $2e=3ac,$ we get $a^2=\frac{4}{5}(b+2c).$ Comparing the coefficients of $y_1^4y_2$ in 
$\phi(R_2),$ 
$$-5a^4c+4a^3e+12a^2bc-3abe-3acd-3b^2c+2de=-5 \lm=-5c^3.$$ 
Using $2e=3ac$ and $a^2=\frac{4}{5}(b+2c),$ we get 
$$\frac{91}{25}b^2+\left(\frac{64}{25}+12\right)bc+\left(\frac{64}{25}+5\right)c^2=0.$$
Thus $91\frac{b^2}{c^2}+364\frac{b}{c}+189=0.$ Since $91x^2+364x+189$ is 
irreducible over $\Q,$ the equation $91x^2+364x+189=0$ has no roots 
in $\Q,$ a contradiction. Hence $c=0.$\\
Let $\phi(R_2)=\lm R_1^\prime.$ Then comparing the coefficients 
of $y_2^3,$ we get $\lm =0.$ Thus $\phi(R_2)=0.$ Hence comparing the coefficients of $y_1^2y_2^2,$ 
we get $e=0.$ Let $\phi(R_3)=\lm_1 y_1R_1^\prime+\lm_2 R_2^\prime.$ Then comparing the coefficients 
of $y_1y_2^3$ and $y_1^3y_2^2,$ we get
\begin{eqnarray}\label{equation1 for lm1 and lm2}
 -\lm_1+4\lm_2&=&0\\
\label{equation2 for lm1 and lm2}
6\lm_1-10\lm_2&=&0. 
\end{eqnarray}
Solving equations (\ref{equation1 for lm1 and lm2}) and (\ref{equation2 for lm1 and lm2}) gives 
$\lm_1=\lm_2=0.$ Thus $\phi(R_3)=0.$ Therefore by Lemma \ref{vanishing of Ri}, $\phi$ is trivial.\\
 $(3)$ Let $\bar\phi:H^*(G_{5,2};\Q) \rightarrow H^*(G_{5,1},\Q)$ be a graded algebra homomorphism. Then 
 $\bar\phi$ is trivial. But there exists a nontrivial map $\bar h:H^*(G_{5,2};\mathbb R) 
 \rightarrow H^*(G_{5,1},\mathbb R).$\\
 {\it Proof:} Let $I=\mathcal I_{5,2},I^\prime=\mathcal I_{5,1}.$ Let $\bar\phi(x_1+I)=ay_1+I^\prime$ and $\bar\phi(x_2+I)=by_1^2+I^\prime.$ Define $\phi:
 \Q[x_1,x_2] \rightarrow \Q[y_1]$ by $\phi(x_1)=ay_1$ and $\phi(x_2)=by_1^2.$ Then $\phi(I)\subset I^\prime$. Since 
 $\deg R_1=8<10$ and $\phi(R_1) \in (y_1^5),$ we get $\phi(R_1)=\phi(x_1^4-3x_1^2x_2+x_2^2)=0.$ Thus 
 $(a^4-3a^2b+b^2)y_1^4=0$ which implies that $a^4-3a^2b+b^2=0.$ If $b=0$ then $a=0$ and hence $\bar \phi $ is trivial. Suppose $b\neq 0.$ Then $\left(\frac{a^2}{b}\right)^2-3\frac{a^2}{b}+1=0.$
 Thus $\frac{a^2}{b}=\frac{3\pm \sqrt{5}}{2} \notin \Q.$ This implies that $\bar \phi $ is trivial. 
 Define $h:\mathbb R[x_1,x_2] \rightarrow 
 \mathbb R[y_1]$ by $h(x_1)=ay_1$ and $h(x_2)=\frac{2a^2}{3\pm \sqrt{5}}y_1^2$. Then $h $ 
 induces a graded algebra homomorphism $\bar h :H^*(G_{5,2};\mathbb R) \rightarrow H^*(G_{5,1},\mathbb R).$ 
 \end{example}

 In section \ref{section4}, we gave examples of nontrivial graded algebra homomorphisms $\bar \phi: H^*(G_{n,k};\Q) \rightarrow H^*(G_{n,1};\Q)$ if $k$ divides $n$ (Example \ref{nontrivial map}). In the next example we show that there do exist infinitely many nontrivial graded algebra homomorphisms $\bar \phi: H^*(G_{n,2};\Q) \rightarrow H^*(G_{n,1};\Q)$ even if $n$ is not even. We need the following proposition for this purpose.

\begin{proposition} \label{rational solutions}
Let $n>1$ be a positive integer. Consider the polynomial  
$$g(c_1,c_2) := \sum_{r+2s=n-1} (-1)^{r+s} {r+s \choose s} c_1^rc_2^s \in
\mathbb{Z}[c_1,c_2],$$
where $c_1$ and $c_2$ are indeterminates.
Then,
\begin{eqnarray*}
&&\{(u,v) \in \mathbb{R}^2 - \{(0,0)\}
: g(u,v) = 0 \}\\
&=&\begin{cases} 
\{(\pm {2 \sqrt{v} \cos (r \pi /n)},v):v>0, r \in \mathbb Z \mbox{ and }e^{2ir \pi/n} \neq \pm{1}\}
\cup\{(0,v):v\neq 0\}& \mbox{ if }2|n\\
\{(\pm {2 \sqrt{v} \cos (r \pi /n)},v):v>0, r \in \mathbb Z \mbox{ and }e^{2ir \pi/n} \neq \pm{1}\}& \mbox{ otherwise}.
\end{cases}
\end{eqnarray*}
Let $A=\{(0,u):u \in \Q-\{0\}\}$, $B=\{(\pm{2u},2u^2):u \in \Q-\{0\} \}$, $C=\{(\pm{3u},3u^2) : u \in \Q-\{0\} \}$, $D=\{(\pm{u},u^2) : u \in \Q-\{0\} \}$. Suppose $(0,0) \neq (u,v) \in \mathbb{Q}^2.$ Then $ g(u,v) = 0$ if and only if 
\begin{equation}
 \label{rational solution of g}
(u,v)\in
\begin{cases}
A & \mbox{ if }n\equiv2,10 \mod 12\\
D & \mbox{ if }n\equiv3,9 \hspace*{0.08in}\mod 12\\
A\cup B & \mbox{ if } n\equiv4,8  \hspace*{0.09in} \mod 12\\
A\cup C\cup D & \mbox{ if } n\equiv6  \hspace*{0.24in}\mod 12\\
A\cup B\cup C\cup D & \mbox{ if }  n\equiv0 \hspace*{0.24in}\mod 12.\\
\end{cases}
\end{equation}
Further, if $n$ is not a multiple of $2$ or $3$, then there are no
non-zero rational solutions of $g(c_1,c_1)$.
\end{proposition}
\begin{proof}
By \cite{sury}, we have 
$$\sum_{s \geq 0} (-1)^s {n-1-s \choose s} (c_1+c_2)^{n-1-2s}(c_1c_2)^s
= c_1^{n-1}+c_1^{n-2}c_2+ \cdots + c_2^{n-1}$$ in $\mathbb{Z}[c_1,c_2]$. 
Let $f(c_1,c_2)= \sum_{s \geq 0} (-1)^s {n-1-s \choose s}  (c_1+c_2)^{n-1-2s}(c_1c_2)^s.$ Then for $ u,v, x,y \in \mathbb{C}, $ we have 
$$g(u,v) = \sum_{s \geq 0} (-1)^{n-1-s} {n-1-s \choose s} u^{n-1-2s}v^s
= (-1)^{n-1} f(x,y),$$ where $x+y=u, xy=v.$ Note that $(u,v) \neq (0,0)$ if and only if $(x,y) \neq (0,0)$. Also,
$g(u,v)=0$ if and only if $f(x,y)=0$. Further, $(x,y)\neq(0,0)$ and $f(x,y) =
x^{n-1}+x^{n-2}y+ \cdots + y^{n-1} = 0$ if and only if $x^n=y^n$ 
and $x \neq y$.\\
Let $v>0 $ and $e^{2ir \pi/n} \neq \pm{1}.$ Let $u=\pm2 \sqrt{v} \cos (r \pi /n).$  Since $x+y=u, xy=v$, we have $x,y=\frac{ u\pm\sqrt{u^2-4v}}{2}.$ Thus $(x,y)=\sqrt{v}(e^{ ir \pi/n},e^{-ir \pi/n}),\sqrt{v}(e^{-ir \pi/n},e^{ir \pi/n})$, $-\sqrt{v}(e^{ir \pi/n},e^{-ir\pi/n})$ and $-\sqrt{v}(e^{-ir\pi/n},e^{ir\pi/n}),$ 
which implies that $x^n=y^n$ and $x \neq y.$ Therefore $f(x,y)=0$ and hence $g(u,v)=0.$ Moreover, if $2|n$, then each monomial in $g(c_1,c_2)$ contains $c_1$, so $g(0,c_2)=0.$\\
Let $g(u,v)=0$ for some $(0,0) \neq (u,v)\in \mathbb{R}^2.$ Then $f(x,y)=0.$ Therefore $x = y e^{2ir \pi/n},$ where $e^{2i r \pi/n} \neq
1$. 
Hence, $u = x+y = y(1+ e^{2ir \pi/n}), v = xy = y^2 e^{2ir \pi/n} \in
\mathbb{R}$ with $e^{2ir \pi/n} \neq 1$.\\
Thus
$$\frac{u^2}{v} = \frac{1+ 2e^{2ir \pi/n}+ e^{4ir \pi/n}}{e^{2ir
\pi/n}} = e^{-2ir \pi/n} + 2 + e^{2ir \pi/n} = 4 \cos^2 (r \pi /n).$$
This implies that\\
1) if $\cos(r\pi/n)\neq0$, then $v>0$ and $u = \pm {2 \sqrt{v} \cos (r \pi /n)}$ with
$e^{2ir \pi/n} \neq \pm{1}$ and \\
2) if $\cos(r\pi/n)=0,$ then $2|n$ and $u=0.$

\noindent Let $(0,0) \neq (u,v) \in \Q^2$ be such that $g(u,v)=0$. Then $u=\pm {2 \sqrt{v} \cos (r \pi /n)}$ or $0$. Let $\cos (r \pi/n) = \cos (m \pi/d)$ where $d|n$ and $(m,d)=1$. 
Since $\sqrt{v} \cos (m \pi/d)\in\Q$  with $0 \neq v \in \Q,$ we have $\cos^2 (m \pi/d)\in \Q$ which means that $\cos (2m
\pi/d)\in \Q$. Thus the only possible values for $\cos (2m \pi/d)$ are $0,1,-1,\frac{1}{2},-\frac{1}{2}$ (see \cite{olm}).
These correspond to $\cos (m \pi/d) = \pm{1/\sqrt{2}}, \pm{1},0,
\pm{\frac{\sqrt{3}}{2}}, \pm{\frac{1}{2}}$ which in turn correspond to $\frac{m}{d} = \frac{1}{4},\frac{3}{4},0,1,\frac{1}{2},\frac{1}{6},\frac{5}{6},\frac{1}{3}$ and $\frac{2}{3}$, respectively. We have omitted the case when $e^{2ir\pi/n}=1,$ i.e. when $\cos(m\pi/d)=1,-1$ and $\frac{m}{d}=0,1$. If $\frac{m}{d}=\frac{1}{2},$ then $2|n$ and $\cos(m\pi/d)=0$.
Thus there are rational solutions as required whenever $n$ is a
multiple of $2$ or $4$ or $6$ or $3$. 
\end{proof}
\begin{example}\label{sury}
For fixed $(u,v) \in \Q^2 -\{(0,0)\}$ as in \eqref{rational solution of g}, define 
$\phi:\Q[x_1,x_2] \rightarrow \Q[y]$ as $\phi(x_1)=uy$ and $\phi(x_2)=vy^2.$ 
Then $\phi$ induces a nontrivial graded algebra homomorphism $\bar \phi:H^*(G_{n,2};\Q) \rightarrow H^*(G_{n,1};\Q)$ and this gives the complete set of non-trivial graded algebra homomorphisms. Further, there are no non-trivial maps $\bar\phi:H^*(G_{n,2};\Q)\to H^*(G_{n,1}:\Q)$ if 
$n$ is not divisible by $2$ or $3$.\\
{\it Proof:} We have $R_1=\sum_{r+2s=n-1} (-1)^{r+s} {r+s \choose s} x_1^rx_2^s.$ 
By Proposition \ref{rational solutions}, $\phi(R_1)=0$ and  thus $\phi$ gives a nontrivial graded algebra homomorphism 
$\bar \phi:H^*(G_{n,2};\Q) \rightarrow H^*(G_{n,1};\Q).$ Also, these are the only possible graded algebra homomorphisms.\\
Suppose $n$ is divisible by neither $2$ nor $3$. Let $\phi(x_1)=u_1y,\phi(x_2)=v_1y^2$ for some $u_1,v_1 \in \Q.$ By Proposition \ref{rational solutions}, $g(u_1,v_1) \neq 0.$ Hence
\[\phi(R_1)= \left(\sum_{r+2s=n-1} (-1)^{r+s} {r+s \choose s} c_1^rc_2^s \right)y^{n-1} = g(u_1,v_1) y^{n-1}\neq 0.\]
Therefore there are no non-trivial graded algebra homomorphisms from $H^*(G_{n,2};\Q)$ to $ H^*(G_{n,1};\Q).$
\end{example}

\section*{Acknowledgements}
\noindent We thank Prof. P. Sankaran for suggesting the problem and many useful discussions. We thank Prof. B. Sury for giving proofs of Proposition \ref{rational solutions} and Example \ref{sury}. We also thank the referees for a careful reading of the manuscript and giving suggestions which have improved our manuscript.

\end{document}